\documentclass[12pt]{article}
\usepackage[table]{xcolor}
\usepackage[margin=1in]{geometry}
\usepackage{amsthm}
\usepackage{amssymb}
\usepackage{amsmath}
\usepackage{graphicx}
\usepackage{mathdots}
\usepackage{mathtools}
\usepackage{url}
\usepackage{color}
\usepackage{framed}
\usepackage{tikz}
\usepackage{tkz-graph}
\usepackage{array}
\usepackage{multirow}

\usepackage{tabularray}
\usepackage{codehigh}

\newcommand{\notleftrightarrow}{\mathrel{\ooalign{$\leftrightarrow$\cr\hidewidth$/$\hidewidth}}}

\newtheorem{theorem}{Theorem}

\newtheorem{lemma}[theorem]{Lemma}

\theoremstyle{definition}

\newtheorem{example}[theorem]{Example}

\newcommand{\mydiameter}{0.5em}
\tikzset{
  mynode/.style={
    circle,
    draw=black,
    fill=black,
    inner sep=0pt,
    minimum size=\mydiameter,
  },
  mynodered/.style={
    mynode,
    draw=red,
    fill=red,
  },
  mynodeempty/.style={
    mynode,
    fill=white,
  },
  myedge/.style={
    very thick,
    gray!70
  },
  myedge0n/.style={myedge},
  myedge34n/.style={myedge},
  myedge17n/.style={myedge},
  myedge21n/.style={myedge},
  myedge49n/.style={myedge},
  myedge0y/.style={myedge,red},
  myedge34y/.style={myedge,blue},
  myedge17y/.style={myedge,green},
  myedge21y/.style={myedge,purple},
  myedge49y/.style={myedge,pink},
}

\title{The Erd\H{o}s unit distance problem for small point sets}

\author{Boris~Alexeev \and Dustin~G.~Mixon\thanks{Department of Mathematics, The Ohio State University, Columbus, Ohio, USA} \thanks{Translational Data Analytics Institute, The Ohio State University, Columbus, Ohio, USA} \and Hans~Parshall}
\date{}

\begin{document}
\maketitle
\begin{abstract}
We improve the best known upper bound on the number of edges in a unit-distance graph on $n$ vertices for each $n\in\{16,\ldots,30\}$.
When $n\leq 21$, our bounds match the best known lower bounds, and we fully enumerate the densest unit-distance graphs in these cases.

On the combinatorial side, our principle technique is to more efficiently generate $\mathcal{F}$-free graphs for a set of forbidden subgraphs $\mathcal{F}$.  On the algebraic side, we are able to determine programmatically whether many graphs are unit-distance, using a custom embedder that is more efficient in practice than tools such as cylindrical algebraic decomposition.
\end{abstract}

\section{Introduction}

A unit-distance graph is a simple graph $G$ for which there exists an injection $f\colon V(G)\to\mathbb{R}^2$ such that $\{u,v\}\in E(G)$ implies $\|f(u)-f(v)\|=1$.
Let $U(n)$ denote the set of unit-distance graphs on $n$ vertices, and let
\begin{equation*}
\label{eq.u(n)}
u(n)
:=\max\{|E(G)|:G\in U(n)\}
\end{equation*}
denote the maximum number of edges in such a graph.  (This is known as $\operatorname{A186705}(n)$ in the On-Line Encyclopedia of Integer Sequences~\cite{OEIS:24}.)
Erd\H{o}s~\cite{Erdos:46} found that an appropriately dilated version of the $\sqrt{n}\times\sqrt{n}$ grid in $\mathbb{R}^2$ delivers the lower bound $u(n)=n^{1+\Omega(1/\log\log n)}$, and he offered a \$500 prize for determining whether there is a matching upper bound.
To date, the best known upper bound is $u(n)=O(n^{4/3})$; see~\cite{Szemeredi:16} and references therein.

We are concerned with estimating $u(n)$ for small values of $n$.
Schade~\cite{Schade:93} obtained the exact value of $u(n)$ for all $n\leq 14$, as well as the complete sets of densest graphs for $n\le 13$.
Schade also obtained lower and upper bounds for $n\leq 30$, some of which were improved by \'{A}goston and P\'{a}lv\"{o}lgyi~\cite{AgostonP:20}, who also determined the exact value of $u(15)$.
As an example of this state of the art, before the present paper, the best known bounds for $n=21$ were
\begin{equation}
\label{eq.example state of the art}
57\le u(21)\le 68.
\end{equation}
Recently, Engel et al.~\cite{EngelEtal:24} searched for point configurations in the so-called \textit{Moser ring} (or the smaller "Moser lattice") to obtain additional lower bounds for $30<n\leq100$, as well as a larger collection of graphs that achieve the best known lower bounds for $n\leq30$.

In the present paper, we improve the best known upper bounds for $16\leq n\leq 30$, resulting in the exact value of $u(n)$ for every $n\leq 21$ (for example, we establish that the left-hand inequality in \eqref{eq.example state of the art} is tight), and we report the complete sets of densest graphs in these cases.

\begin{theorem}[Main Result]\
\label{thm.main result}
\begin{itemize}
\item[(a)]
For each $n\in\{16,\ldots,21\}$, $u(n)$ is given by the following:
\begin{center}
\begin{tabular}{c|ccccccccccccccccccc}
$n$&
$16$&$17$&$18$&$19$&$20$&$21$\\\hline
$u(n)$&
$41$&$43$&$46$&$50$&$54$&$57$
\end{tabular}
\end{center}
\item[(b)]
For each $n\in\{22,\ldots,30\}$, $u(n)$ satisfies the following bounds\footnote{The lower bounds were known previously but are included for easy comparison.  The upper bounds are the improvement; for example, the best previously known bound for $n=22$ was $u(22)\le 72$.}:
\begin{center}
\begin{tabular}{c|ccccccccccccccccccccc}
$n$&$22$&$23$&$24$&$25$&$26$&$27$&$28$&$29$&$30$\\\hline
$u(n)\geq$&$60$&$64$&$68$&$72$&$76$&$81$&$85$&$89$&$93$\\
$u(n)\leq$&$61$&$66$&$72$&$78$&$84$&$90$&$96$&$103$&$110$
\end{tabular}
\end{center}
\item[(c)]
For each $n\in\{0,\ldots,21\}$, the densest graphs in $U(n)$ are enumerated\footnote{Almost all of these graphs were previously discovered by Engel et al.\ in~\cite{EngelEtal:24}. The only exception is a graph on $17$ vertices whose unit-distance embedding does not reside within the \textit{Moser ring} they searched.} in Table~\ref{table:big}.
\end{itemize}
\end{theorem}

The crux of our problem is determining whether a given graph is unit-distance, which amounts to solving a system of polynomial equations over $\mathbb{R}$.
In theory, one could solve such a system using cylindrical algebraic decomposition~\cite{Collins:75}, but since this algorithm exhibits double-exponential runtime (and tends to be slow even for typical real-world instances), this is impractical for graphs on at least $10$ vertices, say.
We sidestep this issue by leveraging recent work by Globus and Parshall~\cite{GlobusP:19}, effectively factoring out much of the (hard) semialgebraic geometry and reducing it to (easy) combinatorics.
Our overall approach (detailed below) is to successively apply three different tests to filter out graphs until only unit-distance graphs remain; see Table~\ref{table:filtering} for how many graphs are filtered out by each test.

First, Globus and Parshall~\cite{GlobusP:19} determined the set $\mathcal{F}$ of $74$ minimal forbidden subgraphs of unit-distance graphs on at most $9$ vertices.
In Section~\ref{section:forbid}, we describe how to enumerate the graphs $\overline{U}(n)$ on $n$ vertices that are $\mathcal{F}$-free.
Notice that the maximum density of such graphs gives an upper bound $\overline{u}(n)$ on $u(n)$.
Using standard graph enumeration tools such as \texttt{nauty}~\cite{McKayP:online}, we are able to compute $u(n)$ for $n\le 15$, but this becomes impractical for $n>15$.
We can continue to extract upper bounds on $\overline{u}(n)$ for larger $n$ by applying an observation due to Schade~\cite{Schade:93} that every dense graph necessarily contains a dense subgraph, together with several tricks.
These enumeration tricks represent the main combinatorial innovation of this paper, allowing us to push this enumeration further than it seems would be possible with other tools.
It turns out that these further upper bounds match the best known lower bound on $u(n)$ when $n\leq 21$.
This proves parts (a) of our main result, and part (b) follows shortly as well.

For part (c), take any $n\leq 21$.
The process described above not only computes $\overline{u}(n)$, but also enumerates all graphs in $\overline{U}(n)$ with $\overline{u}(n)$ edges.
Since $u(n) = \overline{u}(n)$, this means that we have a superset of the set of unit-distance graphs on $n$ vertices and $u(n)$ edges.
We need to identify which of these are unit-distance graphs and find embeddings for each of them.
In Section~\ref{section:tuud}, we leverage another idea due to Globus and Parshall~\cite{GlobusP:19}, namely, \textit{totally unfaithful unit-distance graphs}.
In particular, there are a handful of small unit-distance graphs with two distinguished non-adjacent vertices such that for every unit-distance embedding of the graph, the distinguished vertices are necessarily unit distance apart.
Note that any graph on $n$ vertices and $u(n)$ edges with such a substructure is necessarily not unit-distance; indeed, if it were, then it would still be unit-distance after adding an edge between the distinguished vertices, but then it would have more than $u(n)$ edges, a contradiction.
This rules out most candidates.

\begin{table}[t]
\caption{\label{table:filtering}
Numbers of graphs filtered out by each test in this paper.}
\begin{center}
\begin{tabular}{|r|r|rrrr|}\hline
\rowcolor{gray!30}
$n$ & $u(n)$ & \footnotesize{\begin{tblr}{colspec={l},rowsep=0.5pt}number of graphs\\with $n$ vertices\\and $u(n)$ edges\end{tblr}} & \footnotesize{\begin{tblr}{colspec={l},rowsep=0.5pt}... that are\\ \phantom{...} $\mathcal{F}$-free \\ \phantom{...} \end{tblr}} & \footnotesize{\begin{tblr}{colspec={l},rowsep=0.5pt}... and totally\\\phantom{...} unfaithful-free\\ \phantom{...}\end{tblr}} & \footnotesize{\begin{tblr}{colspec={l},rowsep=0.5pt}... and embeddable\\\phantom{...} (thus counting all\\\phantom{...} unit-distance graphs)\end{tblr}} \\ \hline
0  & 0  & 1                  & 1   & 1  & 1  \\
\rowcolor{gray!10}
1  & 0  & 1                  & 1   & 1  & 1  \\
2  & 1  & 1                  & 1   & 1  & 1  \\
\rowcolor{gray!10}
3  & 3  & 1                  & 1   & 1  & 1  \\
4  & 5  & 1                  & 1   & 1  & 1  \\
\rowcolor{gray!10}
5  & 7  & 4                  & 1   & 1  & 1  \\
6  & 9  & 21                 & 4   & 4  & 4  \\
\rowcolor{gray!10}
7  & 12 & 131                & 1   & 1  & 1  \\
8  & 14 & 1646               & 3   & 3  & 3  \\
\rowcolor{gray!10}
9  & 18 & 34040              & 1   & 1  & 1  \\
10 & 20 & $1.1\cdot 10^{6}$  & 1   & 1  & 1  \\
\rowcolor{gray!10}
11 & 23 & $5.3\cdot 10^{7}$  & 2   & 2  & 2  \\
12 & 27 & $5.5\cdot 10^{9}$  & 1   & 1  & 1  \\
\rowcolor{gray!10}
13 & 30 & $5.8\cdot 10^{11}$ & 1   & 1  & 1  \\
14 & 33 & $7.9\cdot 10^{13}$ & 2   & 2  & 2  \\
\rowcolor{gray!10}
15 & 37 & $2.5\cdot 10^{16}$ & 1   & 1  & 1  \\
16 & 41 & $1.1\cdot 10^{19}$ & 1   & 1  & 1  \\
\rowcolor{gray!10}
17 & 43 & $1.5\cdot 10^{21}$ & 15  & 8  & 7  \\
18 & 46 & $4.7\cdot 10^{23}$ & 84  & 38 & 16 \\
\rowcolor{gray!10}
19 & 50 & $4.2\cdot 10^{26}$ & 17  & 5  & 3  \\
20 & 54 & $4.8\cdot 10^{29}$ & 7   & 1  & 1  \\
\rowcolor{gray!10}
21 & 57 & $2.6\cdot 10^{32}$ & 149 & 19 & 5  \\ \hline
\end{tabular}
\end{center}
\end{table}

At this point, we have exhausted our ideas for using combinatorics to avoid solving polynomial systems, but we still need to find unit-distance embeddings for various graphs.
To this end, we present a custom embeddability solver in Section~\ref{section:embed}, which is the main algebraic innovation of this paper.
This is an algorithm that, given a graph, either returns a unit-distance embedding, or reports ``not unit-distance,'' or reports ``I don't know.''
Unlike more general (slow) tools from semialgebraic geometry, this algorithm is highly specialized to our use case: it applies basic moves from Euclidean geometry and linear algebra to reason about the set of embeddings, and it is designed to perform better for denser graphs.
Accordingly, our algorithm is much faster (often taking about a second for the graphs we consider, though sometimes longer), and it never reports ``I don't know'' for the graphs on $n$ vertices and $u(n)$ edges that survived the filtering from Section~\ref{section:tuud}.
This proves part (c) of our main result.

We conclude in Section~\ref{section:discuss} with a brief discussion.

\section{Filtering with forbidden subgraphs}
\label{section:forbid}

Recently, Globus and Parshall~\cite{GlobusP:19} determined the minimal forbidden subgraphs of $U(n)$ for every $n\leq 9$.
(These were previously known for every $n\leq 7$; see Chilakamarri and Mahoney~\cite{ChilakamarriM:95}.)
Let $\mathcal{F}$ denote this set of $74$ graphs, let $\overline{U}(n)$ denote the set of $\mathcal{F}$-free simple graphs on $n$ vertices, and define
\[
\overline{u}(n)
:=\max\{|E(G)|:G\in\overline{U}(n)\}.
\]
Since $U(n)\subseteq \overline{U}(n)$, we have $u(n)\leq \overline{u}(n)$.
For each $n\leq 23$, we compute $\overline{u}(n)$ (and the set of graphs that achieve this density, except for $n=23$).
For $n\leq 21$, this upper bound happens to match the corresponding lower bound due to Schade~\cite{Schade:93}.
(For $n=22$ and $n=23$, these upper bounds do not match the best known lower bounds, so the situation is slightly more complicated as we describe in Section~\ref{section:tuud}.)

For a given $n,m\in\mathbb{N}$, we are interested in constructing the set $\overline{U}(n,m)$ of $\mathcal{F}$-free simple graphs on $n$ vertices with $m$ edges.
Indeed, if we take $\overline{u}(n,m):=|\overline{U}(n,m)|$, then
\[
\overline{u}(n)=\max\{m:\overline{u}(n,m)>0\},
\]
and the densest graphs in $\overline{U}(n)$ are given by $\overline{U}(n,\overline{u}(n))$.
The naive approach here is to generate all graphs consisting of $n$ vertices and $m$ edges before testing for $\mathcal{F}$-freeness.
This allows one to compute $\overline{u}(n)$ for every $n\leq 10$, though depending on the programming details, the $n=10$ case can take hours.
Alternatively, one might be inclined to use \texttt{nauty}~\cite{McKayP:online} to construct $\overline{U}(n,m)$.
McKay~\cite{McKay:online} suggests adding certain code to \texttt{nauty} in order to support forbidding subgraphs.
This allows one to compute $\overline{u}(n)$ for every $n\leq 15$, though the $n=15$ case takes over a month.
Notably, this already gives the first new value of $u(n)$ in Theorem~\ref{thm.main result}.
In order to approach larger values of $n$, we apply the following observation, as recorded by Schade~\cite{Schade:93}:

\begin{lemma}
\label{lem.schade}
A simple graph with $n\geq1$ vertices and $m$ edges contains an induced subgraph with $n-1$ vertices and at least $\lceil m\cdot\frac{n-2}{n}\rceil$ edges.
\end{lemma}

\begin{proof}
Given such a graph $G$, draw a vertex $v$ uniformly at random from $V(G)$ and delete it to produce a random induced subgraph $H$ on $n-1$ vertices.
Then
\[
\mathbb{E}|E(H)|
=m-\mathbb{E}\operatorname{deg}(v)
=m-\frac{1}{n}\sum_{u\in V(G)}\operatorname{deg}(u)
=m-\frac{1}{n}\cdot 2m.
\]
Finally, the maximum of a random variable is an integer and at least its expectation.
\end{proof}

Given $U(n',m')$ for $n'=n-1$ and each $m'\geq\lceil m\cdot\frac{n-2}{n}\rceil$, one may construct $U(n,m)$ by first considering all possible ways of adding a vertex of degree $m-m'$ to each graph in each $U(n',m')$ and then testing for $\mathcal{F}$-freeness.
Naively implementing this trick allows us to solve the $n=16$ case.
For a smarter implementation, consider the set
\[
\mathcal{F}'
:=\{(F-v,S):F\in\mathcal{F},~v\in V(F),~S=N(v)\}.
\]
Fix $H\in U(n',m')$.
Then for every $(F',S)\in\mathcal{F}'$, we find every copy of $F'$ in $H$ and store the image $T$ of $S$ under the corresponding injection.
The result of this computation is the collection $\mathcal{T}$ of subsets $T\subseteq V(H)$ such that the graph obtained by adding a vertex to $H$ with neighborhood $N\subseteq V(H)$ is $\mathcal{F}$-free if and only if there is no $T\in\mathcal{T}$ such that $N\supseteq T$.
That is, $\mathcal{T}$ is the set of ``bad neighborhoods,'' and we can grow $H$ by adding any vertex whose neighborhood does not contain a bad neighborhood.
This implementation allows us to determine $u(17)$ fairly quickly, and parallelizing the code determines $u(18)$ and $u(19)$ in about 5,000 total CPU hours.

For a more efficient implementation, note that the above logic only requires the minimal subsets $\mathcal{T}'$ in $\mathcal{T}$.
To obtain $\mathcal{T}'$, we first initialize $\mathcal{T}'=\emptyset$.
Then for each $k\geq 1$, we consider each $T\subseteq V(H)$ of size $k$ that does not contain some member of $\mathcal{T}'$.
If for some $(F',S)\in\mathcal{F}'$ there exists a copy of $F'$ in $H$ such that the image of $S$ under the corresponding injection equals $T$, then we add $T$ to $\mathcal{T}'$.
This determines $u(20)$ in about 100 total CPU hours.

For $u(21)$, we consider each $U(n',m')$ with $n'=n-1$ and $m'\geq\lceil m\cdot\frac{n-2}{n}\rceil$ in decreasing order of $m'$.
If $H\in U(n',m')$ has minimum degree $\delta(H)\leq m-m'-2$, then we need not consider $H$.
Indeed, adding a vertex to $H$ of degree $m-m'$ will produce a graph $G\in U(n,m)$ with minimum degree $\delta(H)$ or $\delta(H)+1$, in which case removing this vertex produces a graph $H'$ in either $U(n',m-\delta(H))$ or $U(n',m-\delta(H)-1)$.
Since $m-\delta(H)-1\geq m'+1$, this means $G$ was already considered in a previous iteration.
Next, if $H\in U(n',m')$ has minimum degree $\delta(H)=m-m'-1$, then for similar reasons, we need only consider adding a vertex to $H$ if the new vertex is adjacent to all of the minimum-degree vertices of $H$.
In particular, if $H$ has more than $m-m'$ vertices of minimum degree, then we need not consider $H$.
Furthermore, in the previous paragraph, we need only consider the subsets $T$ that contain all vertices of minimum degree.
Parallelizing this modified implementation determines $u(21)$ in about 1,000 total CPU hours.  A slight extension of this computation also determines $\overline{u}(22)=62$ and $\overline{u}(23)=66$, but these no longer match the best known lower bounds.

This concludes our proof of Theorem~\ref{thm.main result}(a).

\section{Filtering with totally unfaithful unit-distance graphs}
\label{section:tuud}

In this section, we review another important concept due to Globus and Parshall~\cite{GlobusP:19}.
We say a unit-distance graph is \textit{totally unfaithful} if it has a pair of non-adjacent vertices with the property that for every unit-distance embedding of the graph, the vertices are unit distance apart.
Such graphs were used by Globus and Parshall to identify forbidden subgraphs in unit-distance graphs.
Figure~\ref{figure:tuud} illustrates six totally unfaithful graphs, along with a pair of vertices in red that are forced to have unit distance.

\begin{figure}[htb]
\begin{center}
\begin{tikzpicture}[scale=2,remember picture, overlay]
\coordinate (0) at (2.771286446121830945,1.550844034075882168);
\coordinate (1) at (0.000000000000000000,0.000000000000000000);
\coordinate (2) at (2.000000000000000000,0.000000000000000000);
\coordinate (3) at (1.271286446121830945,2.416869437860320815);
\coordinate (4) at (0.7712864461218309450,1.550844034075882168);
\coordinate (5) at (2.271286446121830945,2.416869437860320815);
\coordinate (6) at (0.8333333333333333333,0.5527707983925666415);
\coordinate (7) at (-0.06204688721150238832,0.9980732356833155264);
\coordinate (8) at (0.5000000000000000000,0.8660254037844386468);
\coordinate (9) at (2.833333333333333333,0.5527707983925666415);
\coordinate (10) at (1.937953112788497612,0.9980732356833155264);
\coordinate (11) at (1.333333333333333333,1.418796202177005288);
\coordinate (12) at (0.4379531127884976117,1.864098639467754173);
\coordinate (13) at (2.333333333333333333,1.418796202177005288);
\coordinate (14) at (1.437953112788497612,1.864098639467754173);
\coordinate (15) at (1.771286446121830945,1.550844034075882168);
\coordinate (16) at (1.500000000000000000,0.8660254037844386468);
\coordinate (17) at (1.000000000000000000,0.000000000000000000);
\coordinate (18) at (1.833333333333333333,0.5527707983925666415);
\coordinate (19) at (0.9379531127884976117,0.9980732356833155264);
\coordinate (20) at (1.271286446121830945,0.6848186302914435211);
\coordinate (21) at (2.500000000000000000,0.8660254037844386468);
\coordinate (22) at (1.104619779455164278,0.4453024372907488849);
\coordinate (23) at (2.166666666666666667,1.971567000569571930);
\coordinate (24) at (1.763963292114888542,0.9187323808274457991);
\coordinate (25) at (1.728713553878169055,1.839519168670695050);
\coordinate (26) at (2.000000000000000000,1.732050807568877294);
\coordinate (27) at (1.000000000000000000,1.732050807568877294);
\coordinate (28) at (1.333333333333333333,-0.3132546053918720052);
\coordinate (29) at (0.4379531127884976117,0.1320478318988768796);
\coordinate (30) at (2.271286446121830945,0.6848186302914435211);
\coordinate (31) at (2.333333333333333333,-0.3132546053918720052);
\coordinate (32) at (1.437953112788497612,0.1320478318988768796);
\end{tikzpicture}
\begin{tabular}{ccc}
\begin{tikzpicture}
\draw[myedge] (11) -- (13);
\draw[myedge] (13) -- (18);
\draw[myedge] (13) -- (16);
\draw[myedge] (17) -- (18);
\draw[myedge] (16) -- (17);
\draw[myedge] (8) -- (17);
\draw[myedge] (8) -- (11);
\draw[myedge] (8) -- (16);
\node[mynode,at=(13)] {};
\node[mynode,at=(17)] {};
\node[mynodered,at=(11)] {};
\node[mynodered,at=(18)] {};
\node[mynode,at=(16)] {};
\node[mynode,at=(8)] {};
\end{tikzpicture}
&
\begin{tikzpicture}
\draw[myedge] (2) -- (21);
\draw[myedge] (16) -- (21);
\draw[myedge] (16) -- (27);
\draw[myedge] (8) -- (27);
\draw[myedge] (2) -- (17);
\draw[myedge] (16) -- (17);
\draw[myedge] (8) -- (17);
\draw[myedge] (2) -- (16);
\draw[myedge] (8) -- (16);
\draw[myedge] (26) -- (27);
\draw[myedge] (16) -- (26);
\node[mynodered,at=(21)] {};
\node[mynode,at=(27)] {};
\node[mynode,at=(17)] {};
\node[mynode,at=(2)] {};
\node[mynodered,at=(26)] {};
\node[mynode,at=(16)] {};
\node[mynode,at=(8)] {};
\end{tikzpicture}
&
\begin{tikzpicture}
\draw[myedge] (2) -- (21);
\draw[myedge] (16) -- (27);
\draw[myedge] (8) -- (27);
\draw[myedge] (2) -- (17);
\draw[myedge] (16) -- (17);
\draw[myedge] (8) -- (17);
\draw[myedge] (2) -- (16);
\draw[myedge] (8) -- (16);
\draw[myedge] (21) -- (26);
\draw[myedge] (26) -- (27);
\draw[myedge] (16) -- (26);
\node[mynodered,at=(21)] {};
\node[mynode,at=(27)] {};
\node[mynode,at=(17)] {};
\node[mynode,at=(2)] {};
\node[mynode,at=(26)] {};
\node[mynodered,at=(16)] {};
\node[mynode,at=(8)] {};
\end{tikzpicture}
\\[12pt]
\begin{tikzpicture}
\draw[myedge] (13) -- (16);
\draw[myedge] (11) -- (13);
\draw[myedge] (13) -- (18);
\draw[myedge] (16) -- (17);
\draw[myedge] (8) -- (16);
\draw[myedge] (8) -- (17);
\draw[myedge] (8) -- (11);
\draw[myedge] (1) -- (17);
\draw[myedge] (6) -- (11);
\draw[myedge] (6) -- (18);
\draw[myedge] (1) -- (6);
\draw[myedge] (1) -- (8);
\node[mynode,at=(13)] {};
\node[mynode,at=(16)] {};
\node[mynode,at=(17)] {};
\node[mynodered,at=(11)] {};%
\node[mynodered,at=(18)] {};
\node[mynode,at=(1)] {};
\node[mynode,at=(6)] {};
\node[mynode,at=(8)] {};
\end{tikzpicture}
&
\begin{tikzpicture}
\draw[myedge] (13) -- (16);
\draw[myedge] (11) -- (13);
\draw[myedge] (13) -- (18);
\draw[myedge] (16) -- (17);
\draw[myedge] (8) -- (16);
\draw[myedge] (8) -- (17);
\draw[myedge] (8) -- (11);
\draw[myedge] (1) -- (17);
\draw[myedge] (6) -- (11);
\draw[myedge] (6) -- (18);
\draw[myedge] (1) -- (6);
\draw[myedge] (1) -- (8);
\node[mynode,at=(13)] {};
\node[mynode,at=(16)] {};
\node[mynodered,at=(17)] {};
\node[mynode,at=(11)] {};%
\node[mynodered,at=(18)] {};
\node[mynode,at=(1)] {};
\node[mynode,at=(6)] {};
\node[mynode,at=(8)] {};
\end{tikzpicture}
&
\begin{tikzpicture}
\draw[myedge] (13) -- (16);
\draw[myedge] (11) -- (13);
\draw[myedge] (13) -- (18);
\draw[myedge] (16) -- (17);
\draw[myedge] (8) -- (16);
\draw[myedge] (8) -- (11);
\draw[myedge] (1) -- (17);
\draw[myedge] (6) -- (11);
\draw[myedge] (6) -- (18);
\draw[myedge] (1) -- (6);
\draw[myedge] (1) -- (8);
\node[mynode,at=(13)] {};
\node[mynode,at=(16)] {};
\node[mynodered,at=(17)] {};
\node[mynode,at=(11)] {};%
\node[mynodered,at=(18)] {};
\node[mynode,at=(1)] {};
\node[mynode,at=(6)] {};
\node[mynode,at=(8)] {};
\end{tikzpicture}
\end{tabular}
\end{center}
\caption{\label{figure:tuud}
Some totally unfaithful unit-distance graphs.}
\end{figure}
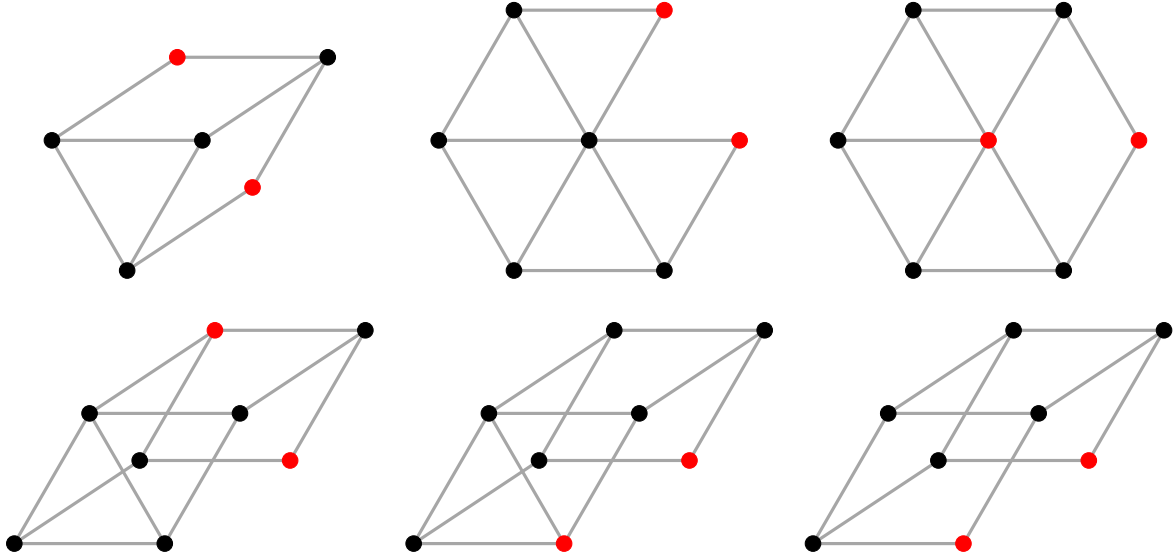

The first five graphs are used by Globus and Parshall, though note that the fourth and fifth graph are the same: two different pairs of vertices are forced to have unit distance in every unit-distance embedding of this graph.
The proof that these graphs are totally unfaithful uses geometric reasoning involving rhombi and equilateral triangles.
The sixth graph is a simplification of the fifth, as the ``cross edge" is not actually needed; in particular, the fifth graph is not a \emph{minimal} totally unfaithful graph.
While Globus and Parshall do not find it necessary to use the sixth graph (as opposed to the fifth), we find it helpful for our purposes.

We use totally unfaithful unit-distance graphs to filter out certain graphs that came from the previous section.
In particular, given a graph on $n$ vertices and $u(n)$ edges, we test to see if it contains a totally unfaithful (not necessarily induced) subgraph for which the distinguished pair of vertices is non-adjacent.
Then one may conclude that this graph is not unit-distance, since otherwise one may add an edge to obtain a unit-distance graph on $n$ vertices with more edges than $u(n)$.

As one can see from Table~\ref{table:filtering}, totally unfaithful graphs are fairly effective at filtering out non-unit-distance graphs, eliminating a particularly large proportion of candidates for $n=21$, where many more non-unit-distance graphs begin to pass the $\mathcal{F}$-free test.

As mentioned earlier, the situation for $n=22$ is interesting.
The densest-known unit-distance graph of this order has $60$ edges.
Using the enumeration techniques from the previous section, we find that $\overline{u}(22) = 62$, and there are exactly two $\mathcal{F}$-free graphs of this size.
It turns out that both of these graphs contain totally unfaithful subgraphs, and so it follows that $u(22)\le 61$.
In particular, $n=22$ is the smallest value of $n$ where $u(n)<\overline{u}(n)$; we expect this to be true for all larger $n$ as well.

We estimate it would take 15,000 total CPU hours to use the techniques from the previous section to enumerate all $\mathcal{F}$-free graphs with $22$ vertices and $61$ edges.
If either all of the resulting candidate graphs can be eliminated using the techniques from this or the following section, or if one of them could be embedded, this would determine $u(22)$.
We attempted to partially enumerate $\mathcal{F}$-free graphs with $22$ vertices and $60$ edges (the densest-known size).
Looking through those graphs, we found at least 1,420 that are $\mathcal{F}$-free and at least 25 that are unit-distance graphs; Engel et al.~\cite{EngelEtal:24} find at least 35 unit-distance graphs.

The techniques from the previous section also determine that $\overline{u}(23)=66$, but again we cannot determine the precise value of $u(23)$.
We also attempted to partially enumerate $\mathcal{F}$-free graphs with $23$ vertices and $64$ edges (the densest-known size).
Looking through those graphs, we found at least 3,177 that are $\mathcal{F}$-free and at least 7 that are unit-distance graphs; Engel et al.~find at least 10 unit-distance graphs.

We conclude this section by noting that Theorem~\ref{thm.main result}(b) follows from using Lemma~\ref{lem.schade} to extrapolate from the fact that $u(22)\le 61$.

\section{Filtering with a custom embeddability solver}
\label{section:embed}

We seek an algorithm that receives a simple graph $G$ and returns one of three things:
\begin{itemize}
\item[(i)]
an injection $f\colon V(G)\to\mathbb{R}^2$ such that $\{u,v\}\in E(G)$ implies $\|f(u)-f(v)\|=1$,
\item[(ii)]
a proof that no such injection exists, or
\item[(iii)]
the statement ``I don't know.''
\end{itemize}
Of course, the algorithm would be more informative if it avoids (iii) for more graphs $G$, but in practice, runtime is also an important consideration.
For example, one could avoid (iii) for every graph by running cylindrical algebraic decomposition~\cite{Collins:75}, but this is impractical due to its double-exponential runtime.
In this section, we present an efficient (yet informative) alternative that applies a series of basic moves from Euclidean geometry and linear algebra.

In what follows, we identify $\mathbb{R}^2$ with $\mathbb{C}$.
Given a simple graph $G$ and a linear operator $A\colon \mathbb{C}^{V(G)}\to\mathbb{C}^{d_A}$, we denote the sentences
\begin{align*}
\exists[f|G]
&=\text{``there exists a unit-distance embedding $f\in\mathbb{C}^{V(G)}$ of $G$''}\\
\exists[f|G,A]
&=\text{``there exists a unit-distance embedding $f\in\operatorname{ker}A$ of $G$''}
\end{align*}
To prove that a graph $G$ does not have a unit-distance embedding, we perform a sequence of logic moves.
We end up getting a lot of mileage out of just four types of logic moves, which we enunciate now and explain later.
The following are expressed in terms of an arbitrary nonnegative integer $i\in\mathbb{N}\cup\{0\}$ and binary string $s\in\{0,1\}^*$:
\begin{itemize}
\item[(L0)]
$\exists[f|G]\Rightarrow\exists[f|G_0,A_0]$
\item[(L1)]
$\neg\exists[f|G_i,A_s]$
\item[(L2)]
$\exists[f|G_i,A_s]\Rightarrow\exists[f|G_{i+1},A_{s0}]$
\item[(L3)]
$\exists[f|G_i,A_s]\Rightarrow\exists[f|G_i,A_{s0}]\vee \exists[f|G_i,A_{s1}]$
\end{itemize}
In practice, we start by applying (L0), and then we proceed by iteratively applying (L1), (L2) and (L3).
Whenever possible, we apply (L1) next.
Otherwise, whenever possible, we apply (L2) next.
Otherwise, whenever possible, we apply (L3) next.
The algorithm terminates if we can logically conclude $\neg\exists[f|G]$, or if there are no more moves available.
In the latter case, we attempt to find an embedding of $G$ that resides in $\operatorname{ker}A_s$ for some $s\in\{0,1\}^*$.

Having established the general structure of the algorithm, we now discuss the details of (L0)--(L3).
For (L0), we identify all $4$-cycles in $G$.
Indeed, if
\[
v_1\leftrightarrow v_2\leftrightarrow v_3\leftrightarrow v_4\leftrightarrow v_1,
\]
then for any unit-distance embedding $f$ of $G$, it necessarily holds that $f(v_1),f(v_2),f(v_3),f(v_4)$ are neighboring vertices of a rhombus, and so $f(v_1)+f(v_3)=f(v_2)+f(v_4)$.
We encode all such constraints as $A_0f=0$, and we put $G_0:=G$.
We will use two different implementations of (L1), which we label (L1a) and (L1b).
For (L1a), we determine whether there exist $v_1,v_2\in V(G_i)$ with $v_1\neq v_2$ such that $f(v_1)=f(v_2)$ for every $f\in\operatorname{ker}A_s$.
If so, then we may conclude $\neg\exists[f|G_i,A_s]$ due to vertex collision.

\begin{example}
Suppose $G=K_4$.
We start by applying (L0).
Since every $4$-tuple of vertices forms a $4$-cycle in $G$, it follows that (the matrix representation of) $A_0$ is a $6\times 4$ matrix whose rows are all permutations of $(+1,-1,+1,-1)$.
Every member of $\operatorname{ker}A_0$ is a scalar multiple of the all-ones vector.
As such, \textit{every} pair of vertices exhibits a vertex collision.
Applying (L1a) then gives $\neg\exists[f|K_4]$, i.e., $K_4$ is not a unit-distance graph.
\qed
\end{example}

For (L1b), we find $v_1,v_2,v_3,v_4\in V(G_i)$ and $\omega\in\mathbb{C}$ such that
\[
v_1\leftrightarrow v_2,
\qquad
v_3\leftrightarrow v_4,
\qquad
|\omega|\neq 1,
\qquad
f(v_1)-f(v_2)=\omega\big(f(v_3)-f(v_4)\big)
\qquad
\forall f\in\operatorname{ker}A_s.
\]
This can be accomplished by performing the following computation for each of the appropriate $v_1,v_2,v_3,v_4\in V(G_i)$:
Take the mapping $B\colon\operatorname{ker}A_s\to\mathbb{C}^2$ defined by
\[
B(f)
=\big(f(v_1)-f(v_2),f(v_3)-f(v_4)\big)
\]
and determine whether $\operatorname{im}B$ is $1$-dimensional.
If so, select any nonzero $(x,y)\in(\operatorname{im}B)^\perp$ and test whether $\omega:=-\overline{y}/\overline{x}$ has unit modulus.
If not, then every embedding $f\in\operatorname{ker}A_s$ of $G_i$ fails to ensure that both $\{f(v_1),f(v_2)\}$ and $\{f(v_3),f(v_4)\}$ have unit distance, and so $\neg\exists[f|G_i,A_s]$.

For (L2), we similarly find $v_1,v_2,v_3,v_4\in V(G_i)$ and $\omega\in\mathbb{C}$ such that
\[
v_1\leftrightarrow v_2,
\qquad
v_3\notleftrightarrow v_4,
\qquad
|\omega|=1,
\qquad
f(v_1)-f(v_2)=\omega\big(f(v_3)-f(v_4)\big)
\qquad
\forall f\in\operatorname{ker}A_s.
\]
This can be accomplished by performing a computation similar to (L1b).
Note that this implies that for every embedding $f\in\operatorname{ker}A_s$ of $G_i$, it holds that $|f(v_3)-f(v_4)|=1$, and so $f$ is also a unit-distance embedding of the graph $G_{i+1}$ obtained by adding the edge $\{v_3,v_4\}$ to $G_i$.
We collect any additional rhombus constraints (as in (L0)) that are introduced by this new edge, and we append them to $A_s$ to get $A_{s0}$.
For (L3), we find $v_1,\ldots,v_6\in V(G_i)$ and a nonzero vector $(a,b,c)\in\mathbb{C}^3$ such that $v_1\leftrightarrow v_2$, $v_3\leftrightarrow v_4$, $v_5\leftrightarrow v_6$, and furthermore,
\[
a\big(f(v_1)-f(v_2)\big)
+b\big(f(v_3)-f(v_4)\big)
+c\big(f(v_5)-f(v_6)\big)
=0
\qquad
\forall f\in\operatorname{ker}A_s.
\]
This can be accomplished by performing a similar computation to the one described for (L2).
Once such a linear relationship is forced, then by the following lemma (which we prove later), we may conclude that one of two additional linear relationships must also hold:

\begin{lemma}
\label{lem.subtle}
Given $a,b,c,x,y,z\in\mathbb{C}$ such that $a x+b y+c z=0$, $|x|=1$, $|y|=1$, and $|z|=1$, then $(x,y)$ necessarily satisfies
\[
(|a|^2+|b|^2-|c|^2+di)a\cdot x
+2|a|^2b\cdot y
=0,
\]
where $d$ is some solution to $d^2=(2|a||b|)^2-(|a|^2+|b|^2-|c|^2)^2$.
\end{lemma}

As such, we append one of the following constraints to $A_s$ to obtain $A_{s0}$, and we append the other constraint to $A_s$ to get $A_{s1}$:
\begin{align*}
\Big(|a|^2+|b|^2-|c|^2\pm i\sqrt{(2|a||b|)^2-(|a|^2+|b|^2-|c|^2)^2}\Big)a\cdot \big(f(v_1)-f(v_2)\big)
&\\
+\quad 2|a|^2b\cdot \big(f(v_3)-f(v_4)\big)
&=0.
\end{align*}
(Note that if $d=0$, then $A_{s1}=A_{s0}$.)

\begin{example}
\label{ex.k3}
Suppose $G=K_3$.
We start by applying (L0).
Since $G$ contains no $4$-cycles, $A_0$ is the $0\times 3$ matrix that represents the trivial linear transformation $\mathbb{C}^{V(G)}\to\{0\}$.
Notice that $A_0$ is not restrictive enough for us to apply (L1a), (L1b), or (L2).
As such, we resort to (L3).
Denote the vertices of $G_0:=G$ by $u_1,u_2,u_3$.
Then every $f\in\operatorname{ker}A_0=\mathbb{C}^{V(G_0)}$ satisfies
\[
\big(f(u_1)-f(u_2)\big)
+\big(f(u_2)-f(u_3)\big)
+\big(f(u_3)-f(u_1)\big)
=0,
\]
and so Lemma~\ref{lem.subtle} gives that every unit-distance embedding $f\in\operatorname{ker}A_0$ of $G_0$ necessarily satisfies one of the following constraints
\[
(1\pm i\sqrt{3})\cdot\big(f(u_1)-f(u_2)\big)
+2\cdot\big(f(u_2)-f(u_3)\big)=0.
\]
As such, we put
\[
A_{00}
=[1+i\sqrt{3},~1-i\sqrt{3},~-2],
\qquad
A_{01}
=[1-i\sqrt{3},~1+i\sqrt{3},~-2].
\]
This produces two leaves to analyze: $\exists[f|G_0,A_{00}]$ and $\exists[f|G_0,A_{01}]$.
However, neither is amenable to (L1)--(L3), and so we attempt to embed $G_0$ with some $f\in\operatorname{ker}A_{00}\cup\operatorname{ker}A_{01}$.
In this case, $\operatorname{ker}A_{00}$ is the span of $(1,1,1)$ and $(1,\omega,\omega^2)$, where $\omega:=e^{2\pi i/3}$.
Geometrically, this means that every $f\in\operatorname{ker}A_{00}$ is a translation, rotation, and dilation of the unit-distance embedding $\frac{1}{\sqrt{3}}(1,\omega,\omega^2)$.
Similarly, by virtue of complex conjugation, every $f\in\operatorname{ker}A_{01}$ is a translation, rotation, and dilation of the \textit{reflected} unit-distance embedding $\frac{1}{\sqrt{3}}(1,\omega^2,\omega)$.
As such, one may obtain a unit-distance embedding by selecting any nonzero member of $\operatorname{ker}A_{00}\cup\operatorname{ker}A_{01}$ and rescaling so that one of the edges has unit distance.
\qed
\end{example}

Lemma~\ref{lem.subtle} is an immediate consequence of the following:

\begin{lemma}
\label{lem.heron-esque}
Given $x,y,z\in\mathbb{C}$ such that $x+y+z=0$, $|x|=a$, $|y|=b$, and $|z|=c$, then $(x,y)$ necessarily satisfies
\[
(a^2+b^2-c^2+id)\cdot x
+2a^2\cdot y
=0,
\]
where $d$ is some solution to $d^2=(2ab)^2-(a^2+b^2-c^2)^2$.
\end{lemma}

Indeed, given $a,b,c,x,y,z\in\mathbb{C}$ that satisfy the hypotheses of Lemma~\ref{lem.subtle}, then a change variables gives
\begin{alignat*}{5}
\tilde{x}
&:=ax,
\qquad
\tilde{y}
&&:=by,
\qquad
\tilde{z}
&&:=cz,\\
\tilde{a}
&:=|a|,
\qquad
\tilde{b}
&&:=|b|,
\qquad
\tilde{c}
&&:=|c|,
\end{alignat*}
which in turn satisfy $\tilde{x}+\tilde{y}+\tilde{z}=0$, $|\tilde{x}|=\tilde{a}$, $|\tilde{y}|=\tilde{b}$, and $|\tilde{z}|=\tilde{c}$.
Thus, Lemma~\ref{lem.heron-esque} implies Lemma~\ref{lem.subtle}.
The proof of Lemma~\ref{lem.heron-esque} is reminiscent of the proof of Heron's formula:

\begin{proof}[Proof of Lemma~\ref{lem.heron-esque}]
First, we consider the degenerate case in which $x$, $y$ and $z$ are collinear.
In this case,
\[
a^2+b^2-c^2
=|x|^2+|y|^2-|x+y|^2
=-2\operatorname{Re}(\overline{x}y)
=-2\overline{x}y.
\]
In particular, $\overline{x}y\in\mathbb{R}$ implies that $(\overline{x}y)^2=|\overline{x}y|^2=(|x||y|)^2$, and so
\[
d^2
=(2ab)^2-(a^2+b^2-c^2)^2
=(2|x||y|)^2-(2\overline{x}y)^2
=0.
\]
Combining these observations then gives
\[
(a^2+b^2-c^2+id)\cdot x
+2a^2\cdot y
=-2\overline{x}yx+2|x|^2y
=0.
\]
It remains to consider the non-degenerate case.
Since $bx$ and $ay$ have the same modulus, there exists $w\in\mathbb{C}$ of unit modulus such that $bxw=ay$.
To determine $w$, it is helpful to interpret $x$, $y$ and $z$ as directed edges in an $(a,b,c)$-triangle.
Let $\theta$ denote the triangle's angle opposite the edge of length $c$.
Then
\[
w
=e^{\pm i(\pi-\theta)}
=-(\cos\theta\pm i\sin\theta),
\]
where the sign is determined by the orientation of our directed triangle.
By the Pythagorean theorem and the law of cosines, we conclude that $w$ is one of
\[
-\cos\theta\mp i\sqrt{1-\cos^2\theta}
=-\frac{1}{2ab}\Big(a^2+b^2-c^2\pm i\sqrt{(2ab)^2-(a^2+b^2-c^2)^2}\Big).
\]
Finally, we clear denominators to obtain
\[
0
=-2a(bxw-ay)
=(a^2+b^2-c^2+id)\cdot x + 2a^2\cdot y.
\qedhere
\]
\end{proof}

Notice that our algorithm up to this point fails to do anything if the girth of the input graph is at least $5$.
Indeed, our algorithm returns ``I don't know'' for every non-unit-distance graph of girth $\geq5$, even though the most informative response would a proof that no embedding exists.
As an example of such ``bad'' input graphs, there exist graphs of girth $\geq5$ with chromatic number $\geq8$ (see Theorem~3.1 in~\cite{Lovasz:68}), while unit-distance graphs necessarily have chromatic number at most $7$~(see the solution to Problem~2.4 in~\cite{Soifer:08}).
However, all such graphs have at least $57$ vertices~(see Table~1 in~\cite{ExooG:19}), which is beyond the scope of this paper.
Still, there are some graphs on few vertices for which our algorithm returns ``I don't know''.

Recall that the embedding in Example~\ref{ex.k3} was determined up to trivial ambiguities by the $A_s$'s.
In general, there will be additional nonlinear degrees of freedom.
For example, in the case of $G=C_4$, we have $A_0=[+1,-1,+1,-1]$, at which point neither of (L1) and (L2) is applicable, and (L3) fails to deliver any new constraints.
As such, we seek an embedding in $\operatorname{ker}A_0$, namely, the span of $(1,1,1,1)$, $(1,1,-1,-1)$ and $(1,-1,-1,1)$.
There are too many degrees of freedom to determine an embedding up to trivial ambiguities, and so we impose an additional constraint.
Notice there is no nontrivial $(a,b)\in\mathbb{C}^2$ such that
\begin{equation}
\label{eq.new constraint}
a\big(f(u_1)-f(u_2)\big)+b\big(f(u_2)-f(u_3)\big)=0
\end{equation}
for every $f\in\operatorname{ker}A_0$.
For this reason, we say the edges $\{u_1,u_2\}$ and $\{u_2,u_3\}$ are \textit{linearly independent}.
Put $a=1$, draw $b$ uniformly from the complex unit circle, and add the constraint \eqref{eq.new constraint} to $A_0$ to get $A_{00}$.
Then $\operatorname{ker}A_{00}$ is spanned by $(1,1,1,1)$ and a vector of the form $(s,it,-s,-it)$ with $s,t\in\mathbb{R}$ determined by $b$.
Every vector in this subspace is a translation, rotation, and dilation of the same unit-distance embedding of $C_4$.
In general, we iteratively introduce random constraints on a pair of linearly independent edges until $\operatorname{ker}A_s$ is $2$-dimensional, at which point an embedding is determined up to translation, rotation, and dilation.
If this embedding is unit-distance, we are done.
Otherwise, we try again some number of times until we give up and declare ``I don't know''.

We ran this algorithm on each of the graphs that survived the filters in the previous two sections.
The algorithm did not return ``I don't know'' for any of these graphs, and the ones that survived this final test are illustrated in Table~\ref{table:big}.
This concludes our proof of Theorem~\ref{thm.main result}(c).

\section{Discussion}
\label{section:discuss}

In this paper, we improved the best known upper bounds on the maximum number of edges in a unit-distance graph on $n$ vertices for various values of $n$.
What follows are a few ideas for subsequent work.
First, it would be interesting to use the custom embeddability solver in Section~\ref{section:embed} to reproduce the forbidden subgraphs established by Globus and Parshall in~\cite{GlobusP:19}, and perhaps even extend their result to unit-distance graphs on $10$ vertices.
Would such an extension make $u(22)$ accessible?
Next, we find that totally unfaithful unit-distance graphs are very effective at pruning candidate graphs, and so it would be valuable to find more examples of such graphs.  
Finally, most of the embeddings in Table~\ref{table:big} reside in what Engel et al.~\cite{EngelEtal:24} refer to as the \textit{Moser ring}, and it would be interesting if looking further within this and related rings could inspire more results related to unit-distance graphs.

\section*{Acknowledgments}

This work was inspired in part by the Polymath16 project on the chromatic number of the plane~\cite{Mixon:online}.
DGM was partially supported by AFOSR FA9550-18-1-0107, NSF DMS 1829955, and NSF DMS 2220304.
HP was partially supported by an AMS-Simons Travel Grant.

\appendix

\section{Unit-distance graphs of maximum density}
\label{appendix.a}

Table~\ref{table:big} presents embeddings of all of the densest unit-distance graphs on at most $21$ vertices.
All but one of these embeddings can be viewed as a slight growth of a smaller embedding, which we illustrate by drawing new edges in color on top of a grayed-out smaller embedding.
(The exception here is the first embedding with $n=21$, which we present in full color.)
For all but one of the embeddings, all of the edges are drawn at an angle that is a $60$-degree rotation of an edge from the \textit{Moser spindle}:

\begin{center}
\begin{tikzpicture}
\draw[myedge0y] (8) -- (17);
\draw[myedge0y] (1) -- (17);
\draw[myedge34y] (1) -- (6);
\draw[myedge0y] (1) -- (8);
\draw[myedge34y] (1) -- (7);
\draw[myedge0y] (16) -- (17);
\draw[myedge0y] (8) -- (16);
\draw[myedge34y] (6) -- (7);
\draw[myedge17y] (4) -- (16);
\draw[myedge34y] (4) -- (6);
\draw[myedge34y] (4) -- (7);
\node[mynode,at=(1)] {};
\node[mynode,at=(16)] {};
\node[mynode,at=(17)] {};
\node[mynode,at=(6)] {};
\node[mynode,at=(8)] {};
\node[mynode,at=(4)] {};
\node[mynode,at=(7)] {};
\end{tikzpicture}
\end{center}
We color code these edges accordingly.
The only exception here is the first embedding with $n=17$, which is obtained by adding a vertex and two edges to the densest unit-distance graph with $n=16$.
Since this is the only graph in Table~\ref{table:big} that is not related to the Moser spindle in this way, it is also the only one that was not already discovered by Engel et al.\ in~\cite{EngelEtal:24}.

Despite our color coding, some of these graphs are not rigid.
For example, our first embedding with $n=6$ is the Minkowski sum of a triangle and an edge, and it exhibits a degree of freedom from the relative angle between the triangle and edge summands.
Similarly, the first embedding with $n=21$ is the Minkowski sum of a triangle $T$ and the wheel graph $W$ on $7$ vertices.
Accordingly,
\[
u(21)
=e(T+W)
=n(T)e(W)+e(T)n(W)
=3\cdot12+3\cdot7
=57.
\]
In other cases, the embedding can be viewed as a Minkowski sum between summands whose relative angle is carefully selected to introduce an extra edge.
(Such extra edges are not possible if one of the summands is a triangle, as in the above examples.)
For example, one may view the embedding with $n=4$ (i.e., the \textit{diamond}) as a Minkowski sum of two edges with a single bonus edge.
Similarly, the embedding with $n=16$ is a Minkowski sum of two copies of the diamond with a single bonus edge.

In addition to the embeddings in Table~\ref{table:big}, we also provide the graph6 codes for the underlying abstract graphs in Table~\ref{table:graph6}.
The graph6 codes have been canonicalized by the \texttt{nauty} program \texttt{labelg}.
These codes also appear in an ancillary file of the arXiv version of this paper.

\newpage

\begin{center}
\begin{longtblr}[
caption = {Unit-distance graphs of maximum density},
label = {table:big},
]{
colspec = {|X[2,c,m] | X[2,c,m] | X[10,c,f] X[10,c,f] X[10,c,f] X[10,c,f]|},
stretch = 0,
rowsep = 6pt,
rowhead = 1,
}
\hline
$n$ & \!$u(n)$ & embedding \\ \hline
$0$ & $0$
& \includegraphics[page=1,scale=0.63]{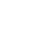}
\\ \hline
$1$ & $0$
& \includegraphics[page=2,scale=0.63]{figure.pdf}
\\ \hline
$2$ & $1$
& \includegraphics[page=3,scale=0.63]{figure.pdf}
\\ \hline
$3$ & $3$
& \includegraphics[page=4,scale=0.63]{figure.pdf}
\\ \hline
$4$ & $5$
& \includegraphics[page=5,scale=0.63]{figure.pdf}
\\ \hline
$5$ & $7$
& \includegraphics[page=6,scale=0.63]{figure.pdf}
\\ \hline
$6$ & $9$
& \includegraphics[page=7,scale=0.63]{figure.pdf}
& \includegraphics[page=8,scale=0.63]{figure.pdf}
& \includegraphics[page=9,scale=0.63]{figure.pdf}
& \includegraphics[page=10,scale=0.63]{figure.pdf}
\\ \hline
$7$ & $12$
& \includegraphics[page=11,scale=0.63]{figure.pdf}
\\ \hline
$8$ & $14$
& \includegraphics[page=12,scale=0.63]{figure.pdf}
& \includegraphics[page=13,scale=0.63]{figure.pdf}
& \includegraphics[page=14,scale=0.63]{figure.pdf}
\\ \hline
$9$ & $18$
& \includegraphics[page=15,scale=0.63]{figure.pdf}
\\ \hline
$10$ & $20$
& \includegraphics[page=16,scale=0.63]{figure.pdf}
\\ \hline
$11$ & $23$
& \includegraphics[page=17,scale=0.63]{figure.pdf}
& \includegraphics[page=18,scale=0.63]{figure.pdf}
\\ \hline
$12$ & $27$
& \includegraphics[page=19,scale=0.63]{figure.pdf}
\\ \hline
$13$ & $30$
& \includegraphics[page=20,scale=0.63]{figure.pdf}
\\ \hline
$14$ & $33$
& \includegraphics[page=21,scale=0.63]{figure.pdf}
& \includegraphics[page=22,scale=0.63]{figure.pdf}
\\ \hline
$15$ & $37$
& \includegraphics[page=23,scale=0.63]{figure.pdf}
\\ \hline
$16$ & $41$
& \includegraphics[page=24,scale=0.63]{figure.pdf}
\\ \hline
$17$ & $43$
& \includegraphics[page=25,scale=0.63]{figure.pdf}
& \includegraphics[page=26,scale=0.63]{figure.pdf}
& \includegraphics[page=27,scale=0.63]{figure.pdf}
& \includegraphics[page=28,scale=0.63]{figure.pdf}
\\
&
& \includegraphics[page=29,scale=0.63]{figure.pdf}
& \includegraphics[page=30,scale=0.63]{figure.pdf}
& \includegraphics[page=31,scale=0.63]{figure.pdf}
\\ \hline
$18$ & $46$
& \includegraphics[page=32,scale=0.63]{figure.pdf}
& \includegraphics[page=33,scale=0.63]{figure.pdf}
& \includegraphics[page=34,scale=0.63]{figure.pdf}
& \includegraphics[page=35,scale=0.63]{figure.pdf}
\\
&
& \includegraphics[page=36,scale=0.63]{figure.pdf}
& \includegraphics[page=37,scale=0.63]{figure.pdf}
& \includegraphics[page=38,scale=0.63]{figure.pdf}
& \includegraphics[page=39,scale=0.63]{figure.pdf}
\\
&
& \includegraphics[page=40,scale=0.63]{figure.pdf}
& \includegraphics[page=41,scale=0.63]{figure.pdf}
& \includegraphics[page=42,scale=0.63]{figure.pdf}
& \includegraphics[page=43,scale=0.63]{figure.pdf}
\\
&
& \includegraphics[page=44,scale=0.63]{figure.pdf}
& \includegraphics[page=45,scale=0.63]{figure.pdf}
& \includegraphics[page=46,scale=0.63]{figure.pdf}
& \includegraphics[page=47,scale=0.63]{figure.pdf}
\\ \hline
$19$ & $50$
& \includegraphics[page=48,scale=0.63]{figure.pdf}
& \includegraphics[page=49,scale=0.63]{figure.pdf}
& \includegraphics[page=50,scale=0.63]{figure.pdf}
\\ \hline
$20$ & $54$ 
& \includegraphics[page=51,scale=0.63]{figure.pdf}
\\ \hline
$21$ & $57$ 
& \includegraphics[page=52,scale=0.63]{figure.pdf}
& \includegraphics[page=53,scale=0.63]{figure.pdf}
& \includegraphics[page=54,scale=0.63]{figure.pdf}
& \includegraphics[page=55,scale=0.63]{figure.pdf}
\\
&
& \includegraphics[page=56,scale=0.63]{figure.pdf}
\\ \hline
\end{longtblr}
\end{center}

\newpage

\begin{table}[t]
\caption{\label{table:graph6}
graph6 codes of unit-distance graphs of maximum density.}
\begin{center}
\scriptsize{
\begin{tabular}{|r|l|}
\hline
$n$ & graph6 code \\ \hline
$0$&\verb=?= \\\hline
$1$&\verb=@= \\\hline
$2$&\verb=A_= \\\hline
$3$&\verb=Bw= \\\hline
$4$&\verb=C^= \\\hline
$5$&\verb=DR{= \\\hline
$6$&\verb=E{Sw= \\
&\verb=EDZw= \\
&\verb=EQlw= \\
&\verb=EElw= \\\hline
$7$&\verb=FoSvw= \\\hline
$8$&\verb=G`iiqk= \\
&\verb=G`iZQk= \\
&\verb=GIISZ{= \\\hline
$9$&\verb=H{dQXgj= \\\hline
$10$&\verb=IISpZATaw= \\\hline
$11$&\verb=J`GWDeNYak_= \\
&\verb=J`GhcpJdQL_= \\\hline
$12$&\verb=KwC[KLQIibDJ= \\\hline
$13$&\verb=L@rLaDBIOidEDJ= \\\hline
$14$&\verb=M_C_?FBNLcTHTaRP_= \\
&\verb=M_DbIo`GWg`RdIah_= \\\hline
$15$&\verb=NGECKA@WW{igRHKpDSW= \\\hline
$16$&\verb=O@iib@`cC_iOAsAi_ioHZ= \\\hline
$17$&\verb=P?CpiPHS@OYAiA@S_UWIY?jK= \\
&\verb=P@Oa@GoQ?d@j@KEoWPOFef?w= \\
&\verb=P?_YQT_K@_r_wG@c_hWDi?ZK= \\
&\verb=P?O`H`OSHaRoq@@I_RWZAAsK= \\
&\verb=P?SaACcK@_q{u??k_LW[aBQK= \\
&\verb=PJPK?CA?gYEF_qEGaaXRAHSK= \\
&\verb=PASaACcG@?rB`xDcAhGTaAYK= \\\hline
$18$&\verb=Q_HG_gT_`?cB?q?hiQ?QTAH^kB?= \\
&\verb=Q`G@O?oDII@YAWD_OHiGs@wqWBo= \\
&\verb=Q?CX@C_CAXOYAg@W`QOIbwINOV?= \\
&\verb=Q_GP@COCGC_LBeBXgwCP`iBDSK_= \\
&\verb=Q?GP?aHP_K?hAaAPtDbCidCpPi?= \\
&\verb=Q_GP?ggC?ZIA@KApSOQCx?{qKF_= \\
&\verb=Q_?oq?`APgPIWW?e_cas@?M]EF_= \\
&\verb=Q_?HGoSDHKBG?J?Ft?ROB_whBPW= \\
&\verb=Q_?HPGdI_oA_?N?oQ_ioQ_ix@VO= \\
&\verb=Q_?P@CgE?oi_?d?R_uOJRiGrNC?= \\
&\verb=Q_?Oh?`E?o_i?t@J?TZOIi?nN_?= \\
&\verb=QGCCIGdX?oq?c@?r?T[K_QPSgaw= \\
&\verb=Q??OP_g?WI_U@bdDHRBaKiIbE@_= \\
&\verb=Qw?G@_K?gDoOO`EBA`XbKaTPDQg= \\
&\verb=QKc?G_HW?G_b`_GqCSkFcSQpGeg= \\
&\verb=QwC?H?W@gA``Ab_XGKXAod@TQAw= \\\hline
$19$&\verb=RJ?GKEB`CGh?AKAIh_`CKC`S`QoaRW= \\
&\verb=R@KCAHD`CG_U?jH_SOI_IQ?sPSoQTW= \\
&\verb=R?CCGx_oE?_Q?b`oKKpgGJ?cOtOPUW= \\\hline
$20$&\verb=S?E@cQHWB?_Q?bDPAcYWKM_BC?OAiW@T[= \\\hline
$21$&\verb=Tsc@IGC@GD?R?S?Wd@A_CK@HG@VM??PRKOUZ= \\
&\verb=TCKx?D?OI?OMCBSA_L?ApA_gEA\EG?PBSCPV= \\
&\verb=TCTWACAG@@CDKC?e?QgQA@OMOq]F??OUcCEj= \\
&\verb=TCS`?H??XIZ?K_Co`CG@JO[?EOSCpGOTSCE\= \\
&\verb=T??_`OhSCSYA@I?c?OyWBEa@c?SIU?Aa[?el= \\\hline
\end{tabular}
}
\end{center}
\end{table}

\end{document}